\documentclass[12pt,twoside]{amsart}
\usepackage{amsfonts,amssymb,amsmath}
\usepackage{amsthm,amscd}
\usepackage{latexsym}
\usepackage[utf8]{inputenc}

\def\Var{\mathrm{Var}}
\def\Der{\mathrm{Der}}
\def\As{\mathrm{As}}
\def\Com{\mathrm{Com}}
\def\Nov{\mathrm{Nov}}
\def\pt{\mathop {\fam 0 pt}\nolimits}
\def\wt{\mathop {\fam 0 wt}\nolimits}

\newtheorem{etheo}{Theorem}
\newtheorem{proposition}{Proposition}
\newtheorem{elemma}{Lemma}
\newtheorem{remark}{Remark}

\newtheorem{eexample}{Example}
\newtheorem{ecorollary}{Corollary}

\begin{document}

\title{Noncommutative Novikov algebras}
\author{P. Kolesnikov, B. Sartayev }
\address{Sobolev Institute of Mathematics, \\
Akad. Koptyug prosp., 4, Novosibirsk, 630090, Russia\\
Suleyman Demirel University, \\ Abylaikhan street 1/1, Kaskelen, Kazakhstan }
\email{pavelsk77@gmail.com\\
baurjai@gmail.com }

\thanks{This research is supported by Russian Science Foundation (project 21-11-00286)}


\maketitle

{\small
\subsection*{Abstract}
The class of Novikov algebras is a popular 
object of study among classical nonassociative algebras.
The generic example of a Novikov algebra may be obtained 
from a differential associative and commutative algebra. 
We consider a more general class of linear algebras
which may be obtained in the same way from not necessarily 
commutative 
associative algebras with a derivation.

{\bf Keywords:} Novikov algebra, derivation, associative algebra

{\bf MSC 2020:} 17D25, 16W50, 16Z10
}

\section{Introduction}

The variety of Novikov algebras is one of the most important 
classes of non-associative algebras. The definition 
of a Novikov algebra goes back to the paper 
\cite{BalNovikov} devoted to the study of
 Poisson brackets of hydrodynamic type, though it appeared earlier in 
\cite{GD79} as a tool for constructing 
 Hamiltonian operators in formal variational calculus.  The term 
``Novikov algebra" was proposed in \cite{Osborn92}. 

Let us briefly show how the axioms of Novikov algebras 
appear in \cite{BalNovikov}. 
We state the corresponding construction in a ``coordinate-free''
form.
Suppose $V$ is a non-associative algebra over a field $\Bbbk $ 
with a bilinear product $\circ $, and $A$ is an associative and commutative algebra with a derivation $d$. Consider the space $A\otimes V$ equipped 
with a bilinear operation $[\cdot,\cdot ]$ given by
\[
[a\otimes v, b\otimes w] = d(a)b\otimes (v\circ w) 
- d(b)a\otimes (w\circ v),
\]
$a,b\in A$, $v,w\in V$. Then $[\cdot ,\cdot ]$
is a Lie bracket if the algebra 
$(V,\circ)$ meets the following relations for all $u,v,w\in V$:
\[
\begin{gathered}
(u\circ v)\circ w  = (u\circ w)\circ v, \\
(u,v,w)_\circ = (v,u,w)_\circ,
\end{gathered}
\]
where 
$(x,y,z)_\circ = (x\circ y)\circ z - x\circ (y\circ z)$. 
These two identities define the variety of Novikov algebras.

In terms of operads \cite{GinzKapr94} this construction 
may be interpreted as follows. Suppose $\mathcal N$ 
is the binary quadratic operad generated by $\mathcal N(2)\simeq \Bbbk S_2$ (one non-symmetric binary operation) 
relative to the defining 
relations representing all identities of degree 3
that hold for the operation $a*b = d(a)b$ on all 
(associative) commutative algebras with a derivation. 
Then the operad $\Nov $ 
of Novikov algebras is isomorphic to $\mathcal N^!$, the Koszul dual to $\mathcal N$. 
This is a well-known fact (see, e.g., \cite{Dzh-nonKoszul}) that $\Nov^!\simeq \Nov^{op}$, 
the anti-isomorphic operad. 
Thus every commutative algebra $A$ with a derivation $d$
is indeed a Novikov algebra relative 
to the operation 
\begin{equation}\label{eq:OsbornConstr}
a\circ b = ad(b),
\end{equation}
$a,b\in A$.

The structure theory of Novikov algebras 
started with \cite{Zelm}, where it was shown 
that a finite-dimensional simple Novikov algebra 
over an algebraically closed field of zero 
characteristic is 1-dimensional. 
Further development of the structure and representation 
theory of Novikov algebras was obtained in 
\cite{Osborn92-1, Osborn94, Xu1996, Xu2001}.

A significant progress in the combinatorial study of Novikov algebras was achieved in \cite{DzhLofwall}, 
where a monomial basis of the free Novikov algebra 
was found. It turned out that the free Novikov 
algebra $\Nov\langle X\rangle $ generated by 
a set $X$ embeds into the free  
commutative differential algebra $\Com\Der\langle X;d\rangle $ 
relative to the operation \eqref{eq:OsbornConstr}.

In \cite{BCZ2017}, the Gr\"obner--Shirshov bases theory 
for Novikov algebras was developed. The constructions 
and proofs of \cite{BCZ2017} are essentially based on the results of \cite{DzhLofwall}, they lead 
to a more general result (a ``Poincar\'e--Birkhoff--Witt Theorem''): 
every Novikov algebra 
may be embedded into an appropriate commutative
differential algebra relative to the operation
\eqref{eq:OsbornConstr}. 

Recent advances in the combinatorial theory of 
Novikov algebras include the study of algebraic 
dependence \cite{DU2021}, nilpotence and solvability 
\cite{ShestZhang, ZhelUmirb}. 

In this paper, we consider non-commutative 
analogues of Novikov algebras whose structure theory 
is richer even over an algebraically closed field 
of zero characteristic. These systems originated in 
\cite{Loday2010}, they were examples of derived 
algebras in \cite{KSO2019}. 
Namely, assume $A$ is an associative (but not necessarily commutative)
algebra with a derivation $d$. Introduce two new 
operations $\prec $ and $\succ $ on the space $A$
as follows:
\[
a\prec b = ad(b),\quad a\succ b = d(a)b,
\]
for $a,b\in A$. 
In \cite{Loday2010}, two identities of degree 3 that hold for such 
operations $\prec$, $\succ $ were found.
It was shown in \cite{KSO2019} that 
there are no more independent identities.

We denote the corresponding variety of algebras
with two operations (as well as the 
binary quadratic operad governing this variety) 
by $D\As$, the class of {\em derived associative} algebras
which is also natural to call {\em noncommutative 
Novikov algebras}. 

We will prove that every $D\As$-algebra 
embeds into an appropriate associative differential algebra,
which is a non-commutative analogue of the ``PBW Theorem'' 
of \cite{BCZ2017}, see also \cite{KS2022}. 
In this paper we present two proofs of the latter statement.
The first one is based on the general theory of derived 
algebras from \cite{KSO2019} which substantially depends 
on the fundamental results of \cite{DzhLofwall} and \cite{BCZ2017}.
The second one 
is completely independent 
and based on the combinatorial Diamond Lemma for associative algebras.
One cannot use the standard Gr\"obner--Shirshov bases 
method (see, e.g., \cite{BokChenBull, BokCKKL}) for such a proof since the compositions 
obtained highly depend on the multiplication 
table in a particular $D\As$-algebra. (So the PBW Theorem in its ``strong'' sense \cite{MikhShest} does not hold here, 
as well as for Novikov algebras.)
However, we develop a ``weight-restricted'' Gr\"obner--Shirshov bases approach 
based on the consideration 
on an appropriate rewriting system (see, e.g., \cite{RS_book}).
In this way, we may get the desired embedding without calculating 
the complete Gr\"obner--Shirshov basis 
of the universal associative differential envelope 
of a $D\As$-algebra.

\section{Derived varieties}

Throughout the paper, we will use the following notations.
An algebra is a linear space $A$ equipped with a family $I$ of bilinear operations
(``multiplications'') $(x,y)\mapsto x\cdot_i y$, $i\in I$. 
A derivation of such an algebra is a linear operator 
$d: A\to A$
such that $d(x\cdot_i y)= d(x)\cdot_i y + x\cdot_i d(y)$
for all $a,y\in A$, $i\in I$.

A variety is a class of algebras with a given 
family of operations which is closed 
with respect to subalgebras, direct products, and 
homomorphic images. By the classical Birkhoff Theorem, 
a variety is defined by a family of identities.
If $\Var $ is a variety of algebras and $X$ is a nonempty set then 
$\Var\langle X\rangle $ stands for the free $\Var $-algebra generated by~$X$.

Denote $\Var\Der$ the variety of $\Var$-algebras equipped with a derivation.
Then $\Var\Der\langle X;d\rangle $ is the free algebra in $\Var\Der$ generated by $X$ (here $d$ denotes the derivation). It is clear that 
$\Var\Der\langle X;d\rangle $ is isomorphic as a $\Var$-algebra 
to $\Var \langle X^{(\omega )} \rangle $, where 
\[
X^{(\omega )} = \bigcup\limits_{n\ge 0} X^{(n)},
\quad X^{(n)} = d^n(X).
\]

Suppose $\Var $ is a variety of algebras defined by multi-linear identities.
So are the classical examples: the variety $\As$ of associative algebras, 
the variety $\Com$ of associative and commutative algebras.
The same symbol $\Var $ will be used for the (symmetric) operad 
governing the variety $\Var $ (see, e.g., \cite{GinzKapr94}).

Given $A\in \Var\Der$, a $\Var$-algebra with a derivation $d$, 
define new operations 
\[
x\prec_i y = x\cdot_i d(y),\quad x\succ_i y = d(x)\cdot_i y,\quad x,y\in A,
\]
for all $i\in I$. The same linear space $A$ equipped with the family of 
operations 
$\prec_i$, $\succ_i$, $i\in I$, is denoted $A^{(d)}$.
Let  $D\Var$ stand for the variety generated by all systems $A^{(d)}$ obtained 
in this way from all $A\in \Var\Der$.
The defining identities of $D\Var$ are exactly those that hold for 
the operations $\prec_i$, $\succ_i$ on all differential $\Var $-algebras.

For example, if $\Var = \Com$ then we obviously have $x\succ y = y\prec x$.
If we denote $x\prec y$ by $xy$ then the following identities hold 
on $D\Com$:
\begin{gather}
    (xy)z -x(yz) = (yx)z-y(xz),\label {eq:LSym} \\
    (xy)z = (xz)y, \label{eq:RCom}
\end{gather}
These are exactly the defining relations of Novikov algebras, and 
it follows from \cite{DzhLofwall} that there are no more 
independent identities that hold in $D\Com$. Thus, $D\Com = \Nov$. 

In general, it was shown in \cite{KSO2019} that the operad governing 
$D\Var $ coincides with the Manin white product of operads (see \cite{GinzKapr94})
$\Var$ and $\Nov$. If $\Var $ is quadratic (i.e., there is a 
basis of defining identities of degree $\le 3$)
then finding the identities of $D\Var$ (i.e., 
the defining relations of the corresponding operad) 
is a standard linear algebra problem. In particular (see \cite{KSO2019}), 
if $\Var = \As$ then $D\As$ is the class of linear spaces 
equipped with two bilinear operations $\prec$ and $\succ $ such that 
\begin{gather}
    x\succ (y\prec z) = (x\succ y)\prec z,\label{eq:LIdent1} \\
    (x\prec y)\succ z - x\succ (y\succ z) = x\prec (y\succ z) - (x\prec y)\prec z. \label{eq:LIdent2}
\end{gather}
The identities \eqref{eq:LIdent1},  \eqref{eq:LIdent2} first appeared
in \cite{Loday2010}
where it was pointed 
out that they hold 
on associative algebras 
with a derivation. It follows from \cite{KSO2019} that 
there are no more independent multilinear identities that hold on $D\As$.

It is straightforward to derive the following identity 
as a corollary of \eqref{eq:LIdent1} and \eqref{eq:LIdent2}:
\begin{equation}\label{eq:LIdent-3-assoc}
(x\prec y, z, b)_\prec = (x,y\succ z, b)_\prec,
\end{equation}
where 
$(a,b,c)_\prec  = (a\prec b)\prec c - a\prec (b\prec c)$.
In a similar way, 
\begin{equation}\label{eq:LIdent-3R-assoc}
(x,y\prec z, b)_\succ = (x,y, z\succ b)_\succ
\end{equation}
holds in every $D\As$-algebra.

\begin{eexample}
Let $A=M_n(\mathbb C)$, $n\ge 3$, and let $d$ be the inner derivation 
$d(x) = [a,x]$, where $a=\mathrm{diag}\,(a_1,\dots, a_n)$.
If $a_1,\dots, a_n\in \mathbb C$ are pairwise different
then $A^{(d)}$ is a simple $D\As$-algebra.
\end{eexample}

Indeed, all matrix unities $e_{ij}$ for $i\ne j$ belong to the image of $d$.
An ideal of $A^{(d)}$ is closed with respect to multiplication by $d(x)$, $x\in A$.
If an ideal $J$ of $A^{(d)}$ contains a nonzero matrix $y$ with $y_{ij}\ne 0$
then 
$y_{ij}e_{ll} = e_{li}ye_{jl}\in J$ for every $l\ne i,j$. Hence, the
diagonal matrix unities along with the identity matrix belong to $J$,
and thus $J=M_n(\mathbb C)$.

\begin{elemma}\label{lem:FreeEmbed}
For every set $X$, the free $D\As$-algebra generated by $X$ is a 
subalgebra of $\As\Der\langle X;d\rangle ^{(d)}$.
\end{elemma}

\begin{proof}
Without loss of generality we may suppose $X$ is an infinite set.
The identity map $x\mapsto x$, $x\in X$, extends to a homomorphism 
$\tau : D\As\langle X\rangle \to \As\Der\langle X;d\rangle ^{(d)}$.

If $\tau(f) =0 $ for a (homogeneous) $f\in D\As\langle X\rangle  $ 
then, obviously, the relation $f=0$ 
holds in $A^{(d)}$ identically for every $A\in \As\Der$. 
If $f$ is a multilinear element of $D\As\langle X\rangle $
then $f=0$ by \cite{KSO2019}, since it follows from \eqref{eq:LIdent1},  \eqref{eq:LIdent2}. 
Otherwise, consider the complete linearization $Lf = Lf(x_1,\dots, x_n)$
 of $f$. Split $Lf$ into summands with a fixed order of elements:
\[
 Lf = \sum\limits_{\sigma \in S_n} f_\sigma (x_{\sigma(1)},\dots, x_{\sigma(n)}),
\]
where 
$f_\sigma (y_1,\dots, y_n)$ is a linear combination of terms obtained from 
the word $y_1\dots y_n$ by some bracketing with operations $\prec $, $\succ $.
For example, if $f=x\prec x$ then $Lf = x_1\prec x_2 + x_2\prec x_1$, 
$f_{\mathrm{id}} = x_1\prec x_2$, $f_{(12)} = x_2\prec x_1$. 
Since $\tau $ does not change the order of variables $x_1,\dots, x_n$, 
each of the multilinear summands $f_\sigma $ belongs to the kernel of $\tau $. 
Hence, $f_\sigma = 0 $ in $D\As\langle X\rangle $ 
for every $\sigma \in S_n$. As $f$ may be obtained from $f_{\mathrm{id}}$ by 
substitution of variables (converse to what was done for linearization), 
we have $f=0$ in $D\As\langle X\rangle $.
Therefore, $\tau $ is injective.
\end{proof}

In this paper 
we will prove that not only free, but every $D\As$-algebra actually embeds into 
an appropriate $\As\Der$-algebra. 

\begin{remark}
The Koszul dual operad $D\As^!$ has the following nilpotence property. 
It is generated by two 
binary operations $\dashv$ and $\vdash$ modulo the relations
\[
\begin{gathered}
((x_1\dashv x_2)\dashv x_3) =0,\quad 
(x_1\vdash (x_2\vdash x_3)) =0, \\
((x_1\vdash x_2)\vdash x_3) =((x_1\vdash x_2)\dashv x_3)=(x_1\dashv (x_2\dashv x_3))= (x_1\vdash (x_2\dashv x_3)) ,\\
(x_1\dashv (x_2\vdash x_3)) =((x_1\dashv x_2)\vdash x_3). \end{gathered}
\]
It is not hard to find the Gr\"obner basis of the corresponding
nonsymmetric operad (as in \cite{DotKhor2010}, see also 
\cite{bremner-dotsenko}):
\[
\begin{gathered}
(x\vdash (x\vdash x)), \ 
((x\dashv x)\dashv x),\ 
(x\dashv (x\dashv (x\vdash x))),\ 
(x\dashv (x\dashv (x\dashv x))) \\
((x\dashv x)\vdash x) -  (x\dashv (x\vdash x)),\ 
(x\vdash (x\dashv x)) - (x\dashv x(x\dashv x)),\\
((x\vdash x)\vdash x)  -  (x\dashv x(x \dashv x)),\ 
((x\vdash x) \dashv x)  -  (x\dashv (x \dashv x)).
\end{gathered}
\]
Hence, $D\As^!(n) = 0$ for $n\ge 4$.
In particular, one may derive that the operad $D\As$ 
is not Koszul in the same way as it was done for $\Nov $ in \cite{Dzh-nonKoszul}.
\end{remark}

\section{Free noncommutative Novikov algebra}

In this section, we describe the image of the free $D\As$-algebra generated by a set $X$ 
in the free associative differential algebra
and introduce the notions which will be used in the sequel. 
Suppose the set of generators $X$ is equipped with a linear 
order.

By Lemma~\ref{lem:FreeEmbed}, $D\As\langle X\rangle  \subset 
\As\Der\langle X;d\rangle \simeq \As\langle X\cup X'\cup X''\cup \dots \rangle $,
where $x^{(n)} = d^n(x)$ for $x\in X$. 
A monomial $u\in \As\Der\langle X;d\rangle $ may 
be uniquely written in the following form: 
\begin{equation}\label{eq:GenericDiffMonomial}
u = x_1^{(n_1)}\dots x_k^{(n_k)},\quad
 x_i\in X,\ n_i \in \mathbb Z_+.
\end{equation}
Let the {\em potential} of $u$ be the following 
polynomial in a formal variable $t$:
\[
\pt(u) = \sum\limits_{j\ge 0} N_j t^j \in \mathbb Z [t],
\]
where $N_j$ is the number of indices $i\in \{1,\dots, k\}$
such that $n_i = j$. For example, 
\[
\pt(x'x''xx'x^{(3)}x''x' ) = t^3+2t^2+3t+1.
\]

Introduce the following order on monomials in 
$\As\Der\langle X;d\rangle $. 
For two such monomials $u$, $v$, let 
$u\ll v$ if either the leading coefficient of $\pt(v)-\pt(u)$ is positive or $\pt(u)=\pt(v)$ but $u<v$ lexicographically 
in $(\mathbb Z_+\times X)^*$, 
assuming that a monomial of the form \eqref{eq:GenericDiffMonomial} is identified with the word
\[
(n_1,x_1)\dots (n_k,x_k) \in (\mathbb Z_+\times X)^*,
\]
and the pairs in $\mathbb Z_+\times X$ are compared lexicographically.

For a polynomial $f\in \As\Der\langle X;d\rangle $, 
let $\bar f$ stand 
for the principal (leading) monomial of $f$ with respect to the 
order $\ll $.

Let us note the following properties of the potential.

\begin{proposition}\label{prop:pt-prop}
{\rm (i)} The length $|u|$ of $u$ is equal to $\pt(u)|_{t=1}$;
{\rm (ii)} For every monomials $u,v\in \As\Der\langle X;d\rangle $ 
we have 
$\pt(uv) = \pt(u)+\pt(v)$;
{\rm (iii)} The order $\ll$ is a  monomial one, i.e., $u_1\ll u_2$ implies 
$u_1v\ll u_2v$ and $vu_1\ll vu_2$ for every monomial~$v$.
%
\end{proposition}

\begin{proof} 
Statements (i) and (ii) are obvious. 
In particular, if $\pt(u)=\pt(v)$ then $u$ and $v$ 
have equal length.
To get (iii), note (ii) and the well-known  monomiality of the lexicographic order.
\end{proof}

Define the {\em weight} of a differential monomial 
$u\in \As\Der\langle X;d\rangle $
as follows:
\[
\wt (u) = \dfrac{d}{dt} \dfrac{\pt(u)}{t} \bigg|_{t=1}
= -N_0 +N_2+2N_3+ \dots , 
\]
for $\pt(u) = N_0 +tN_1+t^2N_2 +\dots $.
Obviously, $\wt(uv)=\wt(u)+\wt(v)$.

\begin{proposition}\label{prop:neg-weight}
If $\wt(u)=-m<0$ then $u = u_1\dots u_m$ with $\wt(u_i)=-1$.
\end{proposition}
 
\begin{proof}
Let $u$ be of the form \eqref{eq:GenericDiffMonomial},
$\wt(u) = -m = n_1+\dots +n_k - k$.
For $m=1$, there is nothing to prove. If $m>1$ then 
there exists a prefix $u_1$ of $u$ such that $\wt(u_1)=-1$. 
Indeed, if $n_1=0$ then $u_1$ is of length~1. If $n_1>0$ 
then consider all prefixes of $u$ consecutively.
Adding a single letter to a word may decrease its weight maximum by~$1$. 
Since the total weight of $u$ is negative and the first 
letter is of nonnegative weight, there should exist a prefix of weight $-1$. Obvious induction completes the proof.
\end{proof}

Note that if  
$u=u_1\dots u_m$, $\wt(u_i)=-1$,
and $u$ contains a letter of nonnegative weight 
(i.e., a derivative of $x$) then this letter belongs to 
a subword $u_i$ of length $>1$.

\begin{proposition}\label{prop:Subword}
If $u$ is of the form \eqref{eq:GenericDiffMonomial},
$\wt(u)=-1$ and there exist $i<j$ such that $n_i,n_j>0$, then 
$u$ contains a proper subword $v$ of length $|v|>1$ and
$\wt(v)=-1$.
\end{proposition} 

\begin{proof}
Let us split $u$ into two subwords $u=w_1w_2$ in such a way that 
$x_i^{(n_i)}$ appears in $w_1$ and $x_j^{(n_j)}$---in $w_2$. 
Then either of these words $w_1$, $w_2$ should have a negative weight
since $\wt(w_1)+\wt(w_2)=\wt(u)=-1$. If $\wt(w_1)<0$ then 
by Proposition~\ref{prop:neg-weight} $w_1$ splits into subwords of weight $-1$,
 and the subword $v$ with $x_i^{(n_i)}$ should be of length $>1$.
The case $\wt(w_2)<0$ is similar.
\end{proof}

Recall the homomorphism $\tau:  D\As\langle X\rangle \to \As\Der\langle X;d\rangle^{(d)} $ from the proof of Lemma~\ref{lem:FreeEmbed}: 
$\tau(x)=x$, $\tau(u\prec v) = \tau(u)d(\tau(v))$, $\tau(u\succ v) = d(\tau(u))\tau(v)$.
This is an injective map.

\begin{etheo}\label{thm:weight-criterion}
The image of $\tau $ coincides with the linear span
of differential monomials $u$ with $\wt(u)=-1$. 
\end{etheo} 

\begin{proof}
It is obvious from the definition of $\tau $ that 
its image consists of $\wt$-homogeneous differential
polynomials of weight $-1$. It is enough to note 
that for every monomial $u\in \As\Der\langle X;d\rangle $ 
the polynomial $d(u)$ is weight-homogeneous and
$\wt(d(u))=\wt(u)+1$. 

Conversely, let $u$ be a monomial of the form \eqref{eq:GenericDiffMonomial}
with $\wt(u)=-1$. 
Let us show by induction on $|u| = k$ that 
$u\in \mathrm{Im}\,\tau $. 

For $|u|=1$ the claim is obvious. 
Assume $|u|=k>1$ and the statement is true for 
all monomials of length $<k$. 

Suppose there are more than one $n_i>0$ in $u$. 
Then by Proposition~\ref{prop:Subword} 
there exists a proper subword $v$ in $u$ 
such that $|v|>1$, $\wt(v)=-1$. 
By induction, 
$v=\tau(g)$ for some $g\in D\As\langle X\rangle $.
Consider $u=u_1vu_2$ as a differential word 
in the alphabet $X\cup \{v\}$: in this extended alphabet, 
the length of $u$ is smaller than $k$. 
Thus by the induction assumption there exists 
$f\in D\As\langle X\cup \{v\} \rangle $
such that $\tau (f) = u\in \As\Der\langle X\cup\{v\}; d\rangle $.
It remains to replace in $f$ the variable $v$ with $g$ to get a desired 
pre-image of $u$ in $D\As\langle X\rangle $.

Suppose there is only one $n_i>0$, i.e., 
\[
u = x_1\dots x_{i-1} x_i^{(k-1)}x_{i+1}\dots x_k.
\]
Consider 
\[
w = 
x_1\prec \big (x_2\prec  \dots \prec \big(x_{i-1}\prec 
( \dots((x_i\succ
 x_{i+1})\succ x_{i+2})\succ \dots \succ x_k)
 \big )\dots \big).
\]
Then all monomials in  $u - \tau (w)$
are of weight $-1$ and 
have potentials of degree $<k-1$, thus contain more than 
one derivative. Therefore, $u-\tau(w)$
belongs to the image of $\tau $ and so is $u$.
\end{proof}

\begin{ecorollary}\label{cor:DAs-basis}
The elements $\tau^{-1}(u)$, $\wt(u)=-1$, form 
a linear basis of $D\As\langle X\rangle $.
\end{ecorollary}


At the level of operads (i.e., for multilinear components of $D\As\langle X\rangle $), 
we obtain the following 

\begin{ecorollary}
For every $n\ge 1$ we have 
$\dim D\As(n) = n!\dim \Nov(n)$, 
so the Manin white product of operads $\As\circ \Nov$ 
coincides with the Hadamard product $\As\otimes \Nov$.
\end{ecorollary}

An intriguing open problem is 
to describe all those binary quadratic operads $\Var $ 
for which $\Var\circ \Nov $ coincides with $\Var \otimes \Nov$.

\section{Embedding into differential algebras}

Let $D$ be a $D\As$-algebra, i.e., a linear space 
equipped with two bilinear operations $\prec $, $\succ$
satisfying \eqref{eq:LIdent1}, \eqref{eq:LIdent2}. 
In this section we will show that there exists an
associative differential algebra $A\in \As\Der $
with a derivation $d$ such that $D\subseteq A^{(d)}$.
This is a noncommutative analogue of the ``PBW Theorem'' 
\cite{BCZ2017} for Novikov algebras.

We present here two proofs of this statement. 
The first one is based on Corollary~\ref{cor:DAs-basis}
and thus depends on the fundamental result 
of \cite{DzhLofwall} on the free Novikov algebra. 
The second proof is completely independent 
of the last result. 

\begin{proposition}\label{prop:CohnLemma}
Let $X$ be a set and let $J$ be an ideal in 
$D\As\langle X\rangle$.
The latter is embedded into $\As\Der\langle X;d\rangle $.
Denote by $I$ the (differential) ideal in 
$\As\Der\langle X;d\rangle $ generated by~$J$. 
Then $I\cap D\As\langle X\rangle  = J$.
\end{proposition}

\begin{proof}
It is enough to show 
$I\cap D\As\langle X\rangle  \subseteq J$
since the inverse embedding is obvious. 
Suppose $f\in I$ is weight-homogeneous: 
all monomials in $f$ are of weight $-1$.
This is a necessary and sufficient condition 
for $f\in D\As\langle X\rangle $.
As $f\in I$, it can be presented 
in the form 
\[
f = \sum\limits_{i=1}^k F_i(g_i, x_1,\dots, x_n), 
\quad g_i\in J,
\]
where each $F_i(t_i,x_1,\dots, x_n)$ is an associative differential polynomial 
in the variables $X\cup \{t_i\}$.
Hence there exists $F(t_1,\dots, t_k, x_1,\dots , x_n)\in \As\langle \widehat X^{(\omega )}\rangle $, 
$\widehat X = X\cup\{t_1, \dots, t_k\}$,
such that 
\[
 f = F(t_1,\dots, t_k, x_1,\dots , x_n)|_{t_i=g_i,\, i=1,\dots , k}. 
\]
Since $\wt (f)=-1$ and $\wt (g_i)=-1$ for all $i=1,\dots, k$, 
the polynomial $F$ is also weight-homogeneous, $\wt (F)=-1$. 
By Theorem~\ref{thm:weight-criterion}, $F\in D\As\langle \widehat X\rangle $, so 
$f = F|_{t_i=g_i} \in J$.
\end{proof}

\begin{etheo}\label{thm:embedding}
For every noncommutative Novikov algebra $D\in D\As$
there exists an associative algebra $A$ with a derivation $d$
such that $D\subseteq A^{(d)}$. 
\end{etheo}

\begin{proof}
Let $D\simeq D\As\langle X\rangle /J$ for some set $X$ 
and ideal $J$. Construct the differential ideal $I$ 
of $\As\Der\langle X;d\rangle $ generated by $J$. 
Then $A = \As\Der\langle X;d\rangle /I$ is a 
differential associative algebra, and the kernel of the map
\[
D\As\langle X\rangle \overset{\subseteq}{\to} 
 \As\Der\langle X;d\rangle \to \As\Der\langle X;d\rangle /I
\]
coincides with $I\cap D\As\langle X\rangle $ which is $J$
by Proposition~\ref{prop:CohnLemma}. 
This map is a homomorphism of $D\As$-algebras.
Therefore, $D$ embeds into $A^{(d)}$.
\end{proof}

\begin{proof}[Alternative proof of Theorem~\ref{thm:embedding}]


Let $X$ be a basis of $D$, then $x\prec y$ and $x\succ y$ for
$x,y\in X$ are linear forms in $X$.

Consider $F = \As\Der\langle X;d\rangle  = \As\langle X^{(\omega )}\rangle $ , 
where $X^{(\omega )} = X\cup X'\cup X''\cup \dots $,
and the ideal $I$ of $F$ generated by all derivatives of 
\[
xy'-x\prec y,\quad x'y - x\succ y,\quad x,y\in X.
\]
Then $F/I$ is the universal enveloping associative differential algebra of $D$. 
The problem is to show that $D$ embeds into $F/I$, i.e., $I\cap \Bbbk X = 0$.

Suppose $X$ is linearly ordered, then the words 
in the alphabet $X^{(\omega )}$ are monomially ordered 
with respect to the order $\ll $ described in Proposition~\ref{prop:pt-prop}.
The explicit set of generators of $I$ as of an ideal 
in the free associative algebra $F$ is given by 
\begin{gather}
R_n(x,y) = x y^{(n)} +\sum\limits_{s=1}^{n-1} \binom{n-1}{s} x^{(s)}y^{(n-s)} - (x\prec y)^{(n-1)}  , \label{eq:R_n}\\ 
L_n(x,y) = x^{(n)} y +\sum\limits_{s=1}^{n-1} \binom{n-1}{s} x^{(n-s)}y^{(s)} - (x\succ y)^{(n-1)}  ,\label{eq:L_n}
\end{gather}
where $x,y\in X$, $n\ge 1$.
Note that $R_n(x,y)$ and $L_n(x,y)$ are weight-homogeneous, 
all monomials are of weight $n-2$. According to the order 
$\ll $, the principal parts of these relations are 
$\bar R_n(x,y)=xy^{(n)}$, $\bar L_n(x,y) = x^{(n)}y$.

Unfortunately, we cannot apply the standard Gr\"obner--Shirshov bases technique 
since the principal parts 
of compositions depend on the particular form of 
$x\prec y$ and $x\succ y$. 
For example, the composition of intersection (see, e.g., \cite{BokChenBull})
of $R_1(x,y) = xy'-x\prec y$ and $L_1(y,z) = y'z-y\succ z$ 
is $(x\prec y)z - x(y\succ z)$, so the choice of a principal part 
is unclear in general.

However, we can draw the desired conclusion if we turn to the Diamond Lemma for rewriting systems
\cite{RS_book} which lies in the foundation of Gr\"obner--Shirshov bases theories.

Consider the rewriting system $\mathcal G$ 
corresponding to 
the generators $X^{(\omega )}$ and rewriting 
rules $R_n(x,y)$, $L_n(x,y)$. This is an oriented 
graph with vertices $F$. 
Two polynomials $f$ and $g$ are connected by an edge 
$f\to g$ if $g$ is obtained from $f$ by eliminating 
a subword of the form $xy^{(n)}$ or $x^{(n)}y$ in a monomial of $f$
by means of $R_n(x,y)$ or $L_n(x,y)$, respectively.
For example, 
\[
\begin{gathered}
 yz''x \to yz'x' - y(z\succ x)' \to (y\prec z)x' - y(z\succ x)', \\
yz''x \to  y'z'x - (y\prec z)'x \to y'(z\succ x) - (y\prec z)'x
\end{gathered}
\]
are edges if $\mathcal G$. A polynomial $f\in F $ belongs to $I$ if and only if 
the vertex $f$ is connected with $0$ by a (non-oriented)
path in the graph $\mathcal G$. 

Since the relations \eqref{eq:R_n}, \eqref{eq:L_n} are weight-homogeneous, 
it is enough to consider only weight-homogeneous vertices in $F$.

Let us state a series of properties of the oriented 
graph $\mathcal G$.

\begin{elemma}\label{lem:Neg1-rewrite}
Let $f\in F $
be a weight-homogeneous polynomial, $\wt(f) = -1$.
Then there exists a chain (oriented path)
\[
f \to \dots \to t,
\]
where $t \in \Bbbk X$.
\end{elemma}

We will denote a chain from $f$ to $g$ by 
\[
 f\to^* g.
\]

\begin{proof}
Let $\pt(f)$ be the maximal potential of monomials in $f$. 
If $\pt(f)=1$ then $f$ itself is a linear form in $X$.
Assume the statement is true for all polynomials $g$
such that $\pt(g)<\pt(f)$. Note that the principal 
parts of $R_n(x,y)$ and $L_n(x,y)$ have the potentials 
greater (even by degree) than all other monomials in these relations. 
Hence, given $f$ as in the statement with $\pt(f)>1$,
there exists a chain $f\to f_1\to \dots \to f_k$ 
where $\pt(f_k)<\pt(f)$: one has to apply rewriting rules to all monomials of maximal potential in $f$.
By induction, $f_k\to ^* t\in \Bbbk X$, so the same is true for~$f$.
\end{proof}

\begin{elemma}\label{lem:Negative-rewrite}
Let $f\in F $
be a weight-homogeneous polynomial, $\wt(f) = -m$, $m>1$.
Then there exists a chain 
\[
f \to ^* t \in \Bbbk X^m.
\]
\end{elemma}

\begin{proof}
It is enough to prove the statement for $f = w$, where $w$ is a monomial
in $X^{(\omega )}$, $\wt(w)=-m$. 
By Proposition~\ref{prop:neg-weight}, $w$ can be represented as $w=u_1\dots u_m$, $\wt(u_i)=-1$. 
Lemma~\ref{lem:Neg1-rewrite} completes the proof: if 
there exist chains $u_i\to^* t_i\in \Bbbk X$ then 
\[
u_1u_2\dots u_m \to^* t_1u_2\dots u_m \to^* \dots \to^* t_1\dots t_m\in \Bbbk X^m.
\]
\end{proof}

\begin{elemma}\label{lem:Zero-weight}
Let $f\in F $
be a weight-homogeneous polynomial, $\wt(f) = 0$.
Then there exists a chain 
\[
f \to ^* t \in \Bbbk (X')^*.
\]
\end{elemma}

\begin{proof}
It is enough to consider the case when $f$ is a monomial
 $w$ in the alphabet $X^{(\omega )}$ of weight zero.

Proceed by induction on the potential of $w$. 
If $\pt(w) = Nt$ then there is nothing to prove. 
Assume $\deg \pt(w)>1$ and the statement holds for all
monomials $u$ such that $\wt(u)=0$ and $\pt(u)<\pt(w)$.

Choose a letter $z^{(n)}$ in $w$ ($z\in X$) with maximal 
$n=\deg \pt(w)$. Then $w = w_1z^{(n)}w_2$, 
$\wt(w_1)+\wt(w_2) = -n+1 <0$. Hence, at least one of $w_1$, $w_2$
has a negative weight. Suppose $\wt(w_2)=-k<0$. Then 
by Lemmas \ref{lem:Neg1-rewrite}, \ref{lem:Negative-rewrite}
there exists a chain $w_2\to^* \sum_i \alpha_i t_{1i}\dots t_{ki}$, 
$t_{ji}\in X$.
So $w$ may be reduced to a linear combination of
$w_1z^{(n)}t_{1i} \dots t_{ki}$,
the latter reduces by $L_n(z,t_{1i})$ to a polynomial $g$ of smaller 
potential, and $g\to^* t\in \Bbbk (X')^*$ by induction.
\end{proof}

Denote by $\mathcal G_m$ the subgraph of $\mathcal G$ spanned by all weight-homogeneous vertices of weight $m$.
Since the relations $R_n(x,y)$ and $L_n(x,y)$ 
are weight-homogeneous, we may conclude 
that if $f,g\in \mathcal G_m$ are connected in $\mathcal G$ (by a non-oriented path) then they are connected in $\mathcal G_m$, 
i.e., there exists such a path through vertices of weight~$m$.
In particular, $\mathcal G_{-1}$ 
is a rewriting system containing all linear polynomials 
$\Bbbk X$.

\begin{elemma}\label{lem:Positive-weight}
Let $f\in F $
be a weight-homogeneous polynomial, $\wt(f) = m>0$.
Then there exists a chain 
\[
f \to^* t \in \Bbbk (X^{(\omega)}\setminus X)^*.
\]
\end{elemma}

\begin{proof} 
Assume there exists a terminal vertex $t$
such that $f\to^* t$ but $t$ contains a letter
from $X=X^{(0)}$ in at least one of its monomials. 
Since the overall weight is positive, this monomial 
should contain also letters of the form $y^{(n)}$, 
$n>0$, and hence there should be an occurence 
of a subword $xy^{(n)}$ or $y^{(n)}x$.
This contradicts terminality of~$t$.
\end{proof}

Recall that a rewriting system is said to be confluent 
if for every connected $f$, $g$ there exists a vertex $h$
and chains
$f\to^* h$ and $g\to^* h$.

A critical pair in
a rewriting system is a pair of edges $w\to f_1$, $w\to f_2$ starting in a vertex $w$ (ambiguity). Such a pair is called convergent 
if there exist a vertex $h$ and two chains 
$f_1\to^* h$, $f_2\to^* h$.
According to the Diamond Lemma, a rewriting system
is confluent if and only if all its critical pairs are convergent. 

\begin{proposition}\label{prop:-1Confl}
The rewriting system $\mathcal G_{-1}$ is confluent.
\end{proposition}

\begin{proof}
For the rewiting system $\mathcal G$ (or $\mathcal G_{-1}$)
it is enough to consider the following ambiguities:
\[
w=w_1xy^{(n)}z w_2\quad \text{or}\quad 
w=w_1x^{(n)}yz^{(m)}w_2,
\quad 
x,y,z\in X.
\]
The critical pairs consist of edges which are defined
by relations $R_n(x,y)$, $L_n(y,z)$ or $L_n(x,y)$, $R_m(y,z)$, 
respectively.

Let us prove the lemma by induction on the potential of vertices. Namely, for a polynomial $p\in \mathbb Z_+[t]$, 
denote by $\mathcal G_{-1}^{p}$ the subgraph of 
$\mathcal G_{-1}$ spanned by all those vertices $f$ of $\mathcal G_{-1}$
for which 
$\pt (u)\le p$ for all monomials $u$ that occur in~$f$
(one has to assume the zero vertex belongs to $\mathcal G_{-1}^p$).
For example, the vertices of the graph $\mathcal G_{-1}^1$ are $\Bbbk X$, and this graph has no edges (hence, 
it is a confluent rewriting system).

\begin{elemma}\label{lem:Trivial_comp_pt}
Two vertices $f$ and $g$ are connected by a non-oriented path in 
$\mathcal G_{-1}^p$ if and only if 
\begin{equation}\label{eq:Trivial_comp_pt}
f-g = \sum\limits_{i=1}^n \alpha_i u_i s_i v_i, \quad u_i,v_i\in (X^{(\omega )})^*,\ \alpha_i\in \Bbbk,
\end{equation}
where $s_i$ are of the form \eqref{eq:R_n} or \eqref{eq:L_n}, 
and $\pt (u_is_iv_i)\le p$.
\end{elemma}

\begin{proof}
The ``only if'' part is obvious. To prove the ``if'' part, proceed by induction on the number $n$ of summands in the right-hand side of 
\eqref{eq:Trivial_comp_pt}. For $n=1$, suppose the monomial
$u_1\bar s_1 v_1$
appears if $g$ with a coefficient $\gamma \in \Bbbk$ (possibly zero).
If $\alpha_1+\gamma =0$ then there is an edge $g\to f$ in 
$\mathcal G_{-1}^p$. If $\gamma= 0$ then there is an edge $f\to g$.
If $\gamma\ne 0$ and $\alpha_1+\gamma\ne 0$ then there is a vertex 
$h$ in $\mathcal G_{-1}^p$ such that $f\to h\leftarrow g$.
The induction step is obvious.
\end{proof}

Let us fix a polynomial $q\in \mathbb Z_+[t]$ and 
assume
$\mathcal G_{-1}^p$ is confluent for all $p<q$.
The purpose is to show that $\mathcal G_{-1}^q$
is also confluent. To that end, check the Diamond Lemma 
condition (convergence of critical pairs).

Case 1: Consider the ambiguity $w = w_1xy^{(n)}z w_2$, 
$\wt (w)=-1$, 
and a critical pair $w\to f_1$, $w\to f_2$, where 
\[
\begin{gathered}
f_1 = w_1 \left ((x\prec y)^{(n-1)} - \sum\limits_{s=1}^{n-1} \binom{n-1}{s} x^{(s)}y^{(n-s)}  \right) z w_2, \\
f_2 = w_1 x 
\left( 
 (y\succ z)^{(n-1)} - \sum\limits_{s=1}^{n-1} \binom{n-1}{s} y^{(n-s)}z^{(s)} 
\right ) w_2 .
\end{gathered}
\]

First, assume $n>1$. Then we may apply a series of rewriting rules 
$L_{n-1}(x\prec y, z)$ to $f_1$ (that is, we apply $L_{n-1}(t,z)$
for each letter $t\in X$ in the linear form $x\prec y$).
For $f_2$, apply $R_{n-1}(x,y\succ z)$. As a result, we obtain 
\begin{multline*}
f_1 \to^* f_{11} = w_1\bigg ( 
((x\prec y)\succ z)^{(n-2)}
- \sum\limits_{j=1}^{n-2} \binom{n-2}{j} (x\prec y)^{(n-1-j)}z^{(j)} \\
- \sum\limits_{s=1}^{n-1} \binom{n-1}{s} x^{(s)}y^{(n-s)} z
\bigg ) w_2
\end{multline*}
\begin{multline*}
    f_2 \to^* f_{21} = 
    w_1  
\bigg( 
 (x\prec (y\succ z))^{(n-2)}
 -\sum\limits_{j=1}^{n-2} \binom{n-2}{j} x^{(j)}(y\succ z)^{(n-1-j)} \\
 - \sum\limits_{s=1}^{n-1} \binom{n-1}{s} xy^{(n-s)}z^{(s)} 
\bigg ) w_2 \\
=
w_1  
\bigg( 
 ((x\prec y)\succ z)^{(n-2)} - (x\succ (y\succ z))^{(n-2)} + 
 ((x\prec y)\prec z)^{(n-2)} \\
 -\sum\limits_{j=1}^{n-2} \binom{n-2}{j} x^{(j)}(y\succ z)^{(n-1-j)} 
 - \sum\limits_{s=1}^{n-1} \binom{n-1}{s} xy^{(n-s)}z^{(s)} 
\bigg ) w_2,
\end{multline*}
the latter equality follows from \eqref{eq:LIdent2}.
Note that $f_{11}$ and $f_{21}$ have smaller potential than $q=\pt(w)$. 
Moreover, 
$f_{11}-f_{21}$ may be presented in the form 
\eqref{eq:Trivial_comp_pt}, where all summands have potentials smaller 
than $q$. Indeed, 
$f_{11}-f_{21} = w_1 g w_2$, where 
\begin{multline*}
g = -\sum\limits_{j\ge 1} \binom{n-2}{j} 
 (x\prec y)^{(n-1-j)}z^{(j)} 
  - \sum\limits_{s\ge 1} \binom{n-1}{s} x^{(s)}y^{(n-s)}z \\
  - ((x\prec y)\prec z)^{(n-2)}
  + \sum\limits_{j\ge 1} \binom{n-2}{j} x^{(j)}
    (y\succ z)^{(n-1-j)}  \\
  + \sum\limits_{s\ge 1} \binom{n-1}{s} xy^{(n-s)}z^{(s)}
  +(x\succ(y\succ z))^{(n-2)}.
\end{multline*}
Let us rewrite the terms
$(x\succ (y\succ z))^{(n-2)}$ and
$((x\prec y)\prec z)^{(n-2)}$
back into the product of derivatives:
\begin{multline*}
g_1 = 
-\sum\limits_{j\ge 1} \binom{n-1}{j} 
 (x\prec y)^{(n-1-j)}z^{(j)}
+ \sum\limits_{j\ge 1} \binom{n-1}{j} 
 x^{(j)} (y\succ z)^{(n-1-j)} \\
+ \sum\limits_{s\ge 1} \binom{n-1}{s} xy^{(n-s)}z^{(s)} 
- \sum\limits_{s\ge 1} \binom{n-1}{s} x^{(s)}y^{(n-s)}z.
\end{multline*}
This rewriting 
involves only summands of smaller potential than 
$xy^{(n)}z$. Hence, $g$ and $g_1$ are connected by a 
non-oriented path in the graph $\mathcal G_{n-3}^{r}$, 
where $r<\pt(xy^{(n)}z)$. Similarly, 
$g_1$ is connected to zero in the same graph
since expanding the remaining terms 
$(y\succ z)^{(n-1-j)}$,
$(x\prec y)^{(n-1-j)}$
into words in $X^{(\omega )}$
leads $g_1$ to vanish.
Thus by Lemma~\ref{lem:Trivial_comp_pt} $f_{11}$ and $f_{21}$ are connected by a 
non-oriented path in $\mathcal G_{-1}^p$
for some $p<q$. Since $\mathcal G_{-1}^p$
is assumed to be confluent, there exist chains 
$f_{11}\to^* h$
and $f_{21}\to^* h$ 
for an appropriate vertex~$h$.
Therefore, the critical pair $f_1\leftarrow w\to f_2$
is convergent.

Next, suppose $n=1$. Since $w=w_1xy'zw_2$ is of weight $-1$, 
we have $\wt(w_1)+\wt(w_2)=1$, i.e., at least one of 
$w_1$, $w_2$ has a positive weight. 

Assume $\wt(w_2)\ge 1$.
By Lemma~\ref{lem:Positive-weight}, there exists a chain 
$w_2\to^* g$, where all monomials in $g$ are  
of the form $a_1'\dots a_r' b^{(m)}v$, $m\ge 2$, $r\ge 0$,
$a_i,b\in X$, $v\in (X^{(\omega )})^*$. 
Then there exist chains in $\mathcal G_{-1}^p$, 
$p<q=\wt (w)$, reducing $f_1$ and $f_2$ 
to linear combinations of monomials
\[
 w_1(x\prec y)z a_1'\dots a_r' b^{(m)} v,\quad 
w_1x(y\succ z)a_1'\dots a_r' b^{(m)} v,
\]
respectively.
Suppose $r>0$. 

Note that there exists a chain in $\mathcal G_{-1}$
of the form 
\[
xyb^{(m)} \to^* -(x,y,b)_\prec ^{(m-2)} +
\sum\limits_{(s,t)\ne (0,0)} \binom{m-2}{s,t} x^{(s)}y^{(t)}b^{(m-s-t)}.
\]
Hence, 
\begin{multline*}
(x\prec y)z a_1'\dots a_r' b^{(m)}  \to^* 
(x\prec y) z_r b^{(m)} \\
\to^* -( (x\prec y), z_r, b)_\prec ^{(m-2)} + \text{terms of smaller potential}, 
\end{multline*}
\begin{multline*}
x(y\succ z a_1'\dots a_r' b^{(m)}  \to ^*
x(y\succ z)_r b^{(m)} \\
\to ^*
-( x, (y\succ z)_r, b)_\prec ^{(m-2)} + \text{terms of smaller potential}, 
\end{multline*}
where 
$z_r = ((\dots( z\prec a_1)\prec  \dots )\prec a_r)$.
Note that $(y\succ z)_r = y\succ z_r$ by \eqref{eq:LIdent1}.
Hence, in the difference of these two expressions 
the associators cancel each other by \eqref{eq:LIdent-3-assoc}.

Therefore, there exist vertices $g_1$, $g_2$ in $\mathcal G_{-1}^p$, $p<q$, such that 
$f_1\to^* g_1$, $f_2\to^* g_2$, and if we expand all remaining operations $\prec$ and $\succ $ in the expression for $g_1-g_2$ then we obtain a differential polynomial 
of smaller potential than $\pt(w)$. By a straightforward
computation one can verify that 
this polynomial is zero. Hence, there exists a vertex 
$h$ in $\mathcal G_{-1}^p$ such that $g_1\to^* h$, $g_2\to^* h$, 
and thus the critical pair $f_1\leftarrow w\to f_2$ is convergent.

The case when $\wt(w_1)\ge 1$ is completely similar,
one has to use \eqref{eq:LIdent-3R-assoc} instead of 
\eqref{eq:LIdent-3-assoc}.

Case 2: Consider the ambiguity $w = w_1x^{(n)}yz^{(m)} w_2$, 
$\wt (w)=-1$, 
and a critical pair $w\to f_1$, $w\to f_2$, where 
\[
\begin{gathered}
f_1 = w_1 \left ((x\succ y)^{(n-1)} - \sum\limits_{s=1}^{n-1} \binom{n-1}{s} x^{(n-s)}y^{(s)}  \right) z^{(m)} w_2, \\
f_2 = w_1 x^{(n)} 
\left( 
 (y\prec z)^{(m-1)} - \sum\limits_{s=1}^{m-1} \binom{m-1}{s} y^{(s)}z^{(m-s)} 
\right ) w_2 .
\end{gathered}
\]
Since $\wt(w)=-1$, we have $\wt(w_1)+\wt(w_2)=-n-m+2$.

First, assume $n,m>1$. Then $\wt(w_1)+\wt(w_2) = -n-m+2<0$, 
so at least one of $w_1$ or $w_2$ has a negative weight. 
Suppose, for example, that $\wt(w_2)<0$. 
Then we may replace $w_2$ with its successor in 
$\mathcal G_{-1}$ from Lemma~\ref{lem:Negative-rewrite}. 
Thus we have two chains in $\mathcal G_{-1}^p$, $p<q=\pt(w)$:
$f_1\to^* g_1$, $f_2\to^* g_2$, where each monomial in $g_i$
is obtained from $f_i$ by means of replacing $w_2$ with 
a linear combination of the form $\sum_j \alpha_j a_jv_j$, $a_j\in X$.

Apply $L_m(z,a_j)$ to $g_1$ and $L_{m-1}(y\succ z,a_j)$ to $g_2$.
Such a rewriting decreases the maximal potential of these polynomials and thus we may ``invert'' one of the initial edges, corresponding to $L_n(x,y)$:
\begin{multline*}
    f_1\to^* g_1 \leftrightarrow 
    \sum\limits_j \alpha_j w_1
    \bigg ( x^{(n)}y(z\succ a)^{(m-1)}-\sum\limits_{s\ge 1} \binom{m-1}{s} 
x^{(n)}y z^{(m-s)}a_j^{(s)} \bigg )\\
\to^*
    h_1 := 
\sum\limits_j
\alpha_j w_1 
\bigg (
x^{(n)} (y\prec (z\succ a_j))^{(m-2)} \\
-\sum\limits_{t\ge 1} \binom{m-2}{t} 
x^{(n)}y^{(t)}(z\succ a_j)^{(m-1-t)}
-\sum\limits_{s\ge 1} \binom{m-1}{s} 
x^{(n)}y z^{(m-s)}a_j^{(s)}
\bigg )v_j
\end{multline*}
Here $\leftrightarrow$ denotes the fact that two vertices 
are connected by a non-oriented path in $\mathcal G_{-1}^p$, 
$p<q$.

Similarly, 
\begin{multline*}
    f_2\to^* g_2 \leftrightarrow h_2:=
\sum\limits_j
\alpha_j w_1 
\bigg (
x^{(n)} ((y\prec z)\succ a_j)^{(m-2)} \\
-\sum\limits_{t\ge 1} \binom{m-2}{t} 
x^{(n)}(y\prec z)^{(m-1-t)}a_j^{(t)}
-\sum\limits_{s\ge 1} \binom{m-1}{s} 
x^{(n)}y^{(s)} z^{(m-s)}a_j
\bigg )v_j
\end{multline*}
It remains to calculate 
$h_1-h_2$, apply \eqref{eq:LIdent2}
to the $(m-2)$th derivative of 
$y\prec (z\succ a_j) - (y\prec z)\succ a_j
= (y\prec z)\prec a_j - y\succ (z\succ a_j)$, 
and then expand back all $\prec $, $\succ $ (this 
expansion  
corresponds to a moving through non-oriented paths 
in $\mathcal G_{-1}^p$, $p<q$, since the terms of maximal 
potential cancel each other in $h_1-h_2$ due to 
\eqref{eq:LIdent2}). 
The polynomial obtained in this way from 
$h_1-h_2$ is zero, so $h_1$, $h_2$
are connected by a non-oriented path in $\mathcal G_{-1}^p$
with $p<q$, and the same is true for $f_1$, $f_2$.  
The inductive assumption for $\mathcal G_{-1}^p$ implies 
the critical pair $f_1\leftarrow w\to f_2$ is convergent.

Next, 
assume $n=m=1$. 
Then $f_1=w_1(x\succ y)z'w_2 \to^* h:=w_1((x\succ y)\prec z) w_2 $ 
and, similarly, 
$f_2\to^* h$ by \eqref{eq:LIdent1}. 
Hence, the critical pair $f_1\leftarrow w\to f_2$ is convergent in this case.

Finally, if either $n=1$, $m>1$ or $n>1$, $m=1$
then one may proceed as in the previous two subcases of Case~2 to prove that the critical pair 
$f_1\leftarrow w\to f_2$ is convergent. 
\end{proof} 

Let us complete the proof of Theorem~\ref{thm:embedding}.
If $f\in \Bbbk X\cap I$ then $f$ is connected with $0$
in $\mathcal G$ and thus in $\mathcal G_{-1}$. 
Since $\mathcal G_{-1}$ is confluent 
there should exist chains starting at $0$ and at $f$
finishing in a single vertex $h$.
 But both $f$ and $0$ are terminal vertices: there are no 
 edges starting at them.
Hence, $f=0$, which proves $\Bbbk X\cap I = 0$.
\end{proof}

\begin{remark}
Our new proof of Theorem~\ref{thm:embedding}
may be adjusted for commutative algebras. 
In this way one may get an alternative proof of 
the embedding Theorem in \cite{BCZ2017}
which is independent from the fundamental 
result of \cite{DzhLofwall} 
on free Novikov algebras.
\end{remark}

\begin{ecorollary}
Let $M$ be an arbitrary nonassociative algebra with 
one operation $(a,b)\mapsto ab$. Then there 
exists $A\in D\As $
such that  $M\subseteq (A,\prec )$. 
The same is true for $(A, \succ) $.
\end{ecorollary}

\begin{proof}
One has to follow the lines of the alternative proof of Theorem~\ref{thm:embedding}: construct the 
universal enveloping associative differential algebra for $M$
as a quotient of $F=\As\langle X^{(\omega )}\rangle $ modulo 
the ideal generated by the relations $R_n(x,y)$ from \eqref{eq:R_n}, 
$n\ge 1$, $x,y\in X$. Here $X$ is a basis 
of $M$ and in $R_n(x,y)$ one should replace $\prec $ with 
the product in $M$.
There are no ambiguities (compositions) of relations \eqref{eq:R_n}, so this is a Gr\"obner--Shirshov basis in the free 
associative algebra $F$
and $M$ embeds into the corresponding quotient.
\end{proof}

It is well known that if $(V,\circ ) $ 
is a Novikov algebra then
its commutator algebra $V^{(-)}$ constructed on the 
same linear space with respect to new operation 
$[x,y] = x\circ y - y\circ x$ is a Lie algebra. 
Every Lie algebra $V^{(-)}$ obtained in this way
meets the identity 
\[
\sum\limits_{\sigma \in S_4}
(-1)^\sigma 
 [x_{\sigma(1)},[x_{\sigma(2)},[x_{\sigma(3)},[x_{\sigma(4)},x_5]]]]=0.
\]
(It was found in \cite{Dzh2005} and re-discovered in \cite{Poins}.)
It is an open problem to find the complete list of independent 
Lie identities that hold on all commutator Novikov algebras.
In the noncommutative case (for $D\As$-algebras), we may find 
the answer.

\begin{ecorollary}
For $A\in D\As$, let $A^{(-)}$ stand for the same linear space equipped with the operation 
$[x,y] = x\prec y - y\succ x$.
Then there are no identities that hold 
on all $A^{(-)}$ for $A\in D\As$.
\end{ecorollary}

\begin{proof}
As in the previous statements, use the monomial order $\ll $ 
from Proposition~\ref{prop:pt-prop}. 
Then for every nonassociative algebra $M$ with a linear basis $X$ one may construct its universal enveloping 
associative differential algebra as a quotient 
of $\As\langle X^{(\omega )}\rangle $
modulo the ideal generated by all derivatives of 
\[
xy' - y'x - xy,\quad x,y\in X.
\]
The principal parts of these derivatives are of the 
form 
$y^{(n)}x$, $n\ge 1$, and there are no ambiguities.

Hence, every nonassociative algebra embeds into 
$A^{(-)}$ for an appropriate $A\in D\As$.
\end{proof}

\subsection*{Data Availability Statement}
Data sharing is not applicable to this article as no new 
data were created or analyzed in this study.


\begin{thebibliography}{99}


\bibitem{BalNovikov}
{\em A. A. Balinskii, S. P. Novikov,} 
Poisson brackets of hydrodynamic type, Frobenius algebras and Lie algebras,
Sov. Math. Dokl., 32, 228--231 (1985).

\bibitem{BokChenBull} 
{\em L. A. Bokut, Y. Chen,} 
Gr\"obner--Shirshov bases and their calculation. 
Bull. Math. Sci., {4(3)}, 325--395 (2014).

\bibitem{BokCKKL}
{\em L. Bokut, Y. Chen, K. Kalorkoti, P. Kolesnikov, V. Lopatkin,}
Gr\"obner--Shirshov bases. Normal forms, combinatorial and decision problems in algebra. 
World Sci. Publ., Hackensack, NJ (2020).

\bibitem{BCZ2017}
{\em L. A. Bokut, Y. Chen, Z. Zhang,}
Gr\"obner--Shirshov bases method for Gelfand--Dorfman--Novikov 
algebras, 
J. Algebra Appl.,  16(1), 1750001, 22 pp.  (2017).

\bibitem{bremner-dotsenko}
{\em M.~R.~Bremner, V.~Dotsenko,}
Algebraic Operads: An Algorithmic Companion.
Chapman and Hall/CRC, 2016.

\bibitem{DotKhor2010}
{\em V.~Dotsenko, A.~Khoroshkin,}
Gr\"obner bases for operads.
Duke Math. J., 153(2) (2010) 363--396.

\bibitem{DU2021}
{\em B. Duisengaliyeva, U. Umirbaev,}
Differential algebraic dependence and Novikov dependence, 
Linear Multilinear Algebra 69(6), 1061--1071 (2021). 

\bibitem{DzhLofwall} 
{\em A. S. Dzhumadil'daev, C. L\"ofwall, }
Trees, free right-symmetric algebras, free Novikov algebras and identities,
Homology, Homotopy Appl., 4(2), 165--190 (2002).

\bibitem{Dzh2005}
{\em A. S. Dzhumadil'daev,}
Special identity for Novikov--Jordan algebras, 
Comm. Algebra 33(5), 1279--1287 (2005).

\bibitem{Dzh-nonKoszul}
{\em A. S. Dzhumadil'daev,}
Codimension growth and non-Koszulity of Novikov operad,
Comm. Algebra, {39}(8), 2943--2952 (2011).

\bibitem{GD79} 
{\em I.~M. Gelfand, I.~Ya. Dorfman,}
Hamilton operators and associated algebraic structures, 
Functional analysis and its application, {13}, no.~4, 13--30 (1979).

\bibitem{GinzKapr94}
{\em V. Ginzburg, M. Kapranov.}
Koszul duality for operads.
Duke Math. J., {76}(1) (1994) 203--272.

\bibitem{KSO2019}
{\em P. S. Kolesnikov, B. Sartayev, A. Orazgaliev,}
Gelfand--Dorfman algebras, derived identities, 
and the Manin product of operads,
J. Algebra {539}, 260--284 (2019).

\bibitem{KS2022}
{\em P. S. Kolesnikov, B. Sartayev,}
On the special identities of Gelfand--Dorfman algebras,
Experimental Math., 2022, DOI: 10.1080/10586458.2022.2041134.

\bibitem{RS_book}
{\em J. W. Klop,} 
Term rewriting systems, 
In: Handbook of Logic in Computer Science, vol.~2, Ch.~1, 1--117. Oxford University Press (1992).

\bibitem{Loday2010}
{\em J.-L. Loday,}
On the operad of associative algebras with derivation,
Georgian Math. J. {17}(2), 347--372 (2010). 

\bibitem{MikhShest}
{\em A. A. Mikhalev, I. P. Shestakov,}
PBW-pairs of varieties of linear algebras,
Comm. Algebra {42}, 667--687 (2014).


\bibitem{Osborn92}
{\em J. M. Osborn,}
Novikov algebras, Nova J. Algebra Geom.
{1}, 1--14 (1992).

\bibitem{Osborn92-1}
{\em J. M. Osborn,}
Simple Novikov algebras with an idempotent, 
Comm. Algebra {20}(9), 2729--2753 (1992).

\bibitem{Osborn94}
{\em J. M. Osborn,}
Infinite-dimensional Novikov algebras of characteristic~0,
J. Algebra {167}, 146--167 (1994).

\bibitem{Poins}
{\em L. Poinsot,}
The solution to the embedding problem of a
(differential) Lie algebra into its Wronskian envelope,
Comm. Algebra {46} (4), 1641--1667 (2018).

\bibitem{ShestZhang}
{\em I. Shestakov, Z. Zhang,}
Solvability and nilpotency of Novikov algebras,
Comm. Algebra 48(12), 5412--5420 (2020). 

\bibitem{Xu1996}
{\em X. Xu,}
On simple Novikov algebras and their irreducible modules,
J. Algebra {185}(3), 905--934 (1996).

\bibitem{Xu2001}
{\em X. Xu,}
Classification of simple Novikov algebras and their irreducible modules of characteristic~0,
J. Algebra {246}(2), 673--707 (2001).

\bibitem{Zelm}
{\em E. I. Zel'manov,}
On a class of local translation invariant Lie algebras,
Sov. Math. Dokl. {35}, 216--218 (1987).

\bibitem{ZhelUmirb}
{\em V. Zhelyabin, U. Umirbaev,}
On the solvability of $\mathbb Z_3$-graded
Novikov algebras,
Symmetry  13, 312 (2021) DOI 10.3390/sym13020312.

\end{thebibliography}
\end{document}